\documentclass[runningheads]{llncs}
\usepackage[T1]{fontenc}
\usepackage{graphicx}
\usepackage{hyperref}
\usepackage{color}

\usepackage{caption}
\usepackage{pgf,tikz}
\usetikzlibrary{math,decorations.pathreplacing,calligraphy}
\pgfdeclarelayer{background}
\pgfdeclarelayer{foreground}
\pgfdeclarelayer{front}
\pgfsetlayers{background,main,foreground,front}

\usepackage{pgfplots}

\usepackage[dvipsnames]{xcolor}

\usepackage{amsmath,amssymb}
\usepackage{dsfont}

\usepackage{enumitem}
\def\RMlabel{\upshape({\itshape \roman*\,})}

\newcommand{\cF}{\mathcal{F}}
\newcommand{\cM}{\mathcal{M}}
\newcommand{\cP}{\mathcal{P}}
\newcommand{\cQ}{\mathcal{Q}}
\newcommand{\cR}{\mathcal{R}}
\def\PP{\Pr}

\let\E=\EE
\let\epsilon=\varepsilon
\let\eps=\epsilon
\DeclareMathOperator{\tr}{tr}

\renewcommand{\le}{\leqslant}

\renewcommand{\ge}{\geqslant}
\renewcommand{\geq}{\geqslant}

\newcommand{\RM}{\mathbb{RM}}

\newtheorem{observation}[theorem]{Observation}

\begin{document}
\title{Homogeneous substructures in random ordered uniform matchings\thanks{The first author was supported in part by Simons Foundation Grant MPS-TSM-00007551. The second author was supported in part by Narodowe Centrum Nauki, grant 2020/37/B/ST1/03298. The third author was supported in part by the AGH University of Krakow grant no.\ 16.16.420.054, funded by the Polish Ministry of Science and Higher Education. The fourth author was supported in part by Narodowe Centrum Nauki, grant 2024/53/B/ST1/00164.}}
\titlerunning{Homogeneous substructures in random ordered uniform matchings}

%
\author{Andrzej Dudek\inst{1} \and
Jaros\l aw Grytczuk\inst{2} \and
Jakub Przyby{\l}o\inst{3} \and
Andrzej Ruci\'nski\inst{4}}
\authorrunning{A.~Dudek et al.}
%
\institute{Department of Mathematics, Western Michigan University, Kalamazoo, MI, USA\\
\email{andrzej.dudek@wmich.edu}
\and
Faculty of Mathematics and Information Science, Warsaw University of Technology, Warsaw, Poland\\
\email{j.grytczuk@mini.pw.edu.pl}
\and
AGH University of Krakow, Krakow, Poland\\
\email{jakubprz@agh.edu.pl}
\and
Department of Discrete Mathematics, Adam Mickiewicz University, Pozna\'n, Poland\\
\email{rucinski@amu.edu.pl}
}
\maketitle              
\begin{abstract}
An ordered $r$-uniform matching of size $n$ is a collection of $n$ pairwise disjoint $r$-subsets of a linearly ordered set of $rn$ vertices. For $n=2$, such a matching is called an \emph{$r$-pattern}, as it represents one of $\tfrac12\binom{2r}r$ ways two disjoint edges may intertwine. Given a  set $\cP$ of $r$-patterns, a \emph{$\cP$-clique} is a matching with all pairs of edges belonging to $\cP$.
In this paper we determine the order of magnitude of the size of a largest $\cP$-clique in a \emph{random} ordered $r$-uniform matching 
 for several sets $\cP$, including all sets of size $|\cP|\le2$ and  the set $\cR^{(r)}$ of all $2^{r-1}$ $r$-partite $r$-patterns.
\end{abstract}

\section{Introduction}
We study an Erd\H os-Szekeres-type problem for random ordered $r$-uniform matchings, that is, families consisting of $n$ disjoint $r$-element subsets of a given linearly ordered set of size $rn$.
Before stating our results we present necessary background.

\subsection{Matchings, words, and patterns}
A hypergraph on a linearly ordered vertex set is called \emph{ordered},  and it is \emph{$r$-uniform} if all its edges are of the same size $r$. Two ordered hypergraphs are \emph{order-isomorphic} if there is a usual isomorphism between them preserving their linear orders.
An \emph{ordered $r$-matching} of \emph{size} $n$ is an ordered $r$-uniform hypergraph consisting of $n$ pairwise disjoint edges (and no isolated vertices). For fixed $r$ and $n$, there are precisely
$\frac{(rn)!}{(r!)^n\, n!}$ distinct ordered $r$-matchings of size $n$. The family of all of them will be denoted by $\cM^{(r)}_n$. For instance, the family $\cM^{(2)}_2$ consists of three distinct matchings depicted in Figure \ref{ESZ1}.

\begin{figure}[ht]
	\captionsetup[subfigure]{labelformat=empty}
	\begin{center}
		
		\scalebox{0.80}
		{
			\centering
			\begin{tikzpicture}
				[line width = .5pt,
				vtx/.style={circle,draw,black,very thick,fill=black, line width = 1pt, inner sep=0pt},
				]
				
				\coordinate (0) at (0.5,0) {};
				\node[vtx] (1) at (1,0) {};
				\node[vtx] (2) at (2,0) {};
				\node[vtx] (3) at (3,0) {};
				\node[vtx] (4) at (4,0) {};
				\coordinate (5) at (4.5,0) {};
				\draw[line width=0.3mm, color=lightgray]  (0) -- (5);
				\fill[fill=black, outer sep=1mm]  (1) circle (0.1) node [below] {$1$};
				\fill[fill=black, outer sep=1mm]  (2) circle (0.1) node [below] {$2$};
				\fill[fill=black, outer sep=1mm]  (3) circle (0.1) node [below] {$3$};
				\fill[fill=black, outer sep=1mm]  (4) circle (0.1) node [below] {$4$};
				\draw[line width=0.5mm, color=black, outer sep=2mm] (2) arc (0:180:0.5);
				\draw[line width=0.5mm, color=black, outer sep=2mm] (4) arc (0:180:0.5);
				
				\coordinate (0) at (5.5,0) {};
				\node[vtx] (1) at (6,0) {};
				\node[vtx] (2) at (7,0) {};
				\node[vtx] (3) at (8,0) {};
				\node[vtx] (4) at (9,0) {};
				\coordinate (5) at (9.5,0) {};
				\draw[line width=0.3mm, color=lightgray]  (0) -- (5);
				\fill[fill=black, outer sep=1mm]  (1) circle (0.1) node [below] {$1$};
				\fill[fill=black, outer sep=1mm]  (2) circle (0.1) node [below] {$2$};
				\fill[fill=black, outer sep=1mm]  (3) circle (0.1) node [below] {$3$};
				\fill[fill=black, outer sep=1mm]  (4) circle (0.1) node [below] {$4$};
				\draw[line width=0.5mm, color=black, outer sep=2mm] (3) arc (0:180:0.5);
				\draw[line width=0.5mm, color=black, outer sep=2mm] (4) arc (0:180:1.5);
				
				\coordinate (0) at (10.5,0) {};
				\node[vtx] (1) at (11,0) {};
				\node[vtx] (2) at (12,0) {};
				\node[vtx] (3) at (13,0) {};
				\node[vtx] (4) at (14,0) {};
				\coordinate (5) at (14.5,0) {};
				\draw[line width=0.3mm, color=lightgray]  (0) -- (5);
				\fill[fill=black, outer sep=1mm]  (1) circle (0.1) node [below] {$1$};
				\fill[fill=black, outer sep=1mm]  (2) circle (0.1) node [below] {$2$};
				\fill[fill=black, outer sep=1mm]  (3) circle (0.1) node [below] {$3$};
				\fill[fill=black, outer sep=1mm]  (4) circle (0.1) node [below] {$4$};
				\draw[line width=0.5mm, color=black, outer sep=2mm] (3) arc (0:180:1);
				\draw[line width=0.5mm, color=black, outer sep=2mm] (4) arc (0:180:1);
				
			\end{tikzpicture}
		}
		
	\end{center}
	\caption{Three distinct $2$-matchings of size two.}
	\label{ESZ1}
	
\end{figure}
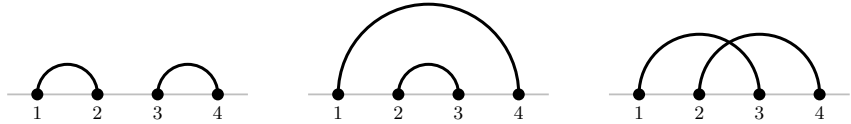

A convenient way to represent ordered $r$-matchings is in terms of \emph{words}. One simply fixes an alphabet of size $n$ and assigns different letters to different edges. Then, every vertex is replaced by the letter assigned to the edge it is contained in, so the obtained word contains each letter exactly $r$ times. In other words,  For instance, if $r=3$ and $M=\{\{1,2,4\}, \{3,10,12\},\{5,7,11\}$, $\{6,8,9\}\}$,
then, two examples of   words representing $M$ are  $AABACDCDDBCB$ and $CCECDFDFFEDE$. 
Just for convenience, to represent a given matching, we typically choose a word in which the first occurrences of letters appear alphabetically.
We will often use the power notation  $W^k$ to write concisely the word
$
W\;W\;\cdots W
$
consisting of $k$ repetitions of the same word $W$. For instance, $ABABAB=(AB)^3$ and $ABBBABAAAA=AB^3ABA^4$.

Ordered $r$-matchings of size $n=2$ are called \emph{$r$-patterns}. For each $r\geqslant 2$, there are exactly $\tfrac12\binom{2r}r$ of them. The set of all $r$-patterns is denoted by $\cP^{(r)}$ (instead of $\cM_2^{(r)}$). For instance, in letter notation, $\cP^{(2)}=\{AABB,\;ABBA,\;ABAB\}$ (cf.  Figure \ref{ESZ1}).
For a fixed pattern $P$, a \emph{$P$-clique} is an ordered matching in which every pair of edges forms pattern $P$. In Figure \ref{ESZ2} one can see three $P$-cliques of size four corresponding to the three $2$-patterns in $\cP^{(2)}$.

\begin{figure}[ht]
	\captionsetup[subfigure]{labelformat=empty}
	\begin{center}
		
		\scalebox{0.80}
		{
			\centering
			\begin{tikzpicture}
				[line width = .5pt,
				vtx/.style={circle,draw,black,very thick,fill=black, line width = 1pt, inner sep=0pt},
				]
				
				\coordinate (0) at (0.25,0) {};
				\node[vtx] (1) at (0.5,0) {};
				\node[vtx] (2) at (1,0) {};
				\node[vtx] (3) at (1.5,0) {};
				\node[vtx] (4) at (2,0) {};
				\node[vtx] (5) at (2.5,0) {};
				\node[vtx] (6) at (3,0) {};
				\node[vtx] (7) at (3.5,0) {};
				\node[vtx] (8) at (4,0) {};
				\coordinate (9) at (4.25,0) {};
				\draw[line width=0.3mm, color=lightgray]  (0) -- (9);
				\fill[fill=black, outer sep=1mm]  (1) circle (0.1) node [below] {$A$};
				\fill[fill=black, outer sep=1mm]  (2) circle (0.1) node [below] {$A$};
				\fill[fill=black, outer sep=1mm]  (3) circle (0.1) node [below] {$B$};
				\fill[fill=black, outer sep=1mm]  (4) circle (0.1) node [below] {$B$};
				\fill[fill=black, outer sep=1mm]  (5) circle (0.1) node [below] {$C$};
				\fill[fill=black, outer sep=1mm]  (6) circle (0.1) node [below] {$C$};
				\fill[fill=black, outer sep=1mm]  (7) circle (0.1) node [below] {$D$};
				\fill[fill=black, outer sep=1mm]  (8) circle (0.1) node [below] {$D$};
				\draw[line width=0.5mm, color=black, outer sep=2mm] (2) arc (0:180:0.25);
				\draw[line width=0.5mm, color=black, outer sep=2mm] (4) arc (0:180:0.25);
				\draw[line width=0.5mm, color=black, outer sep=2mm] (6) arc (0:180:0.25);
				\draw[line width=0.5mm, color=black, outer sep=2mm] (8) arc (0:180:0.25);
				
				\coordinate (0) at (5.25,0) {};
				\node[vtx] (1) at (5.5,0) {};
				\node[vtx] (2) at (6,0) {};
				\node[vtx] (3) at (6.5,0) {};
				\node[vtx] (4) at (7,0) {};
				\node[vtx] (5) at (7.5,0) {};
				\node[vtx] (6) at (8,0) {};
				\node[vtx] (7) at (8.5,0) {};
				\node[vtx] (8) at (9,0) {};
				\coordinate (9) at (9.25,0) {};
				\draw[line width=0.3mm, color=lightgray]  (0) -- (9);
				\fill[fill=black, outer sep=1mm]  (1) circle (0.1) node [below] {$A$};
				\fill[fill=black, outer sep=1mm]  (2) circle (0.1) node [below] {$B$};
				\fill[fill=black, outer sep=1mm]  (3) circle (0.1) node [below] {$C$};
				\fill[fill=black, outer sep=1mm]  (4) circle (0.1) node [below] {$D$};
				\fill[fill=black, outer sep=1mm]  (5) circle (0.1) node [below] {$D$};
				\fill[fill=black, outer sep=1mm]  (6) circle (0.1) node [below] {$C$};
				\fill[fill=black, outer sep=1mm]  (7) circle (0.1) node [below] {$B$};
				\fill[fill=black, outer sep=1mm]  (8) circle (0.1) node [below] {$A$};
				\draw[line width=0.5mm, color=black, outer sep=2mm] (5) arc (0:180:0.25);
				\draw[line width=0.5mm, color=black, outer sep=2mm] (6) arc (0:180:0.75);
				\draw[line width=0.5mm, color=black, outer sep=2mm] (7) arc (0:180:1.25);
				\draw[line width=0.5mm, color=black, outer sep=2mm] (8) arc (0:180:1.75);
				
				\coordinate (0) at (10.25,0) {};
				\node[vtx] (1) at (10.5,0) {};
				\node[vtx] (2) at (11,0) {};
				\node[vtx] (3) at (11.5,0) {};
				\node[vtx] (4) at (12,0) {};
				\node[vtx] (5) at (12.5,0) {};
				\node[vtx] (6) at (13,0) {};
				\node[vtx] (7) at (13.5,0) {};
				\node[vtx] (8) at (14,0) {};
				\coordinate (9) at (14.25,0) {};
				\draw[line width=0.3mm, color=lightgray]  (0) -- (9);
				\fill[fill=black, outer sep=1mm]  (1) circle (0.1) node [below] {$A$};
				\fill[fill=black, outer sep=1mm]  (2) circle (0.1) node [below] {$B$};
				\fill[fill=black, outer sep=1mm]  (3) circle (0.1) node [below] {$C$};
				\fill[fill=black, outer sep=1mm]  (4) circle (0.1) node [below] {$D$};
				\fill[fill=black, outer sep=1mm]  (5) circle (0.1) node [below] {$A$};
				\fill[fill=black, outer sep=1mm]  (6) circle (0.1) node [below] {$B$};
				\fill[fill=black, outer sep=1mm]  (7) circle (0.1) node [below] {$C$};
				\fill[fill=black, outer sep=1mm]  (8) circle (0.1) node [below] {$D$};
				\draw[line width=0.5mm, color=black, outer sep=2mm] (5) arc (0:180:1);
				\draw[line width=0.5mm, color=black, outer sep=2mm] (6) arc (0:180:1);
				\draw[line width=0.5mm, color=black, outer sep=2mm] (7) arc (0:180:1);
				\draw[line width=0.5mm, color=black, outer sep=2mm] (8) arc (0:180:1);
				
			\end{tikzpicture}
		}
		
	\end{center}
	
	\caption{Cliques of size $4$ corresponding to the three patterns in $\cP^{(2)}$.}
	\label{ESZ2}
	
\end{figure}
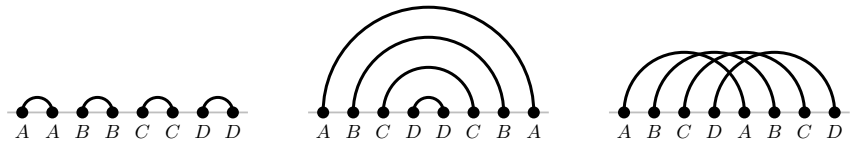

   Inspired by the celebrated Erd\H os-Szekeres Theorem for permutations, it was proved in \cite{DGR-match} (cf. \cite{DGR_LATIN}) and \cite{HJW2019} that every ordered $2$-matching of size $n$ contains, for \emph{some} $P\in \cP^{(2)}$, a $P$-clique of size $\lceil n^{1/3}\rceil$ (and this is optimal). In \cite{AJKS} and \cite{SZ}, this was generalized to arbitrary $r$: every ordered $r$-matching of size $n$ contains, for \emph{some} $P\in \cP^{(r)}$, a $P$-clique of size, roughly, $ n^{1/(2^r-1)}$.

 Our goal is to study more complex substructures, and in a random setting. Let $\cP \subseteq \cP^{(r)}$ be a fixed subset of $r$-patterns. An ordered $r$-matching $M$ is called a \emph{$\cP$-clique} if every pair of edges in $M$ forms a pattern belonging to~$\cP$. Let $z_{\cP}(M)$ denote the largest size of a $\cP$-clique in $M$ and let $\RM^{(r)}_{n}$ denote a random ordered $r$-matching of size $n$, that is, an $r$-matching picked uniformly at random out of the set of all ordered $r$-matchings on the set $[rn]$. Using the introduced notation, our main problem can be stated as follows.

 {\bf Problem.} 
 	\emph{For a fixed $r\geqslant 2$ and for a given subset of patterns $\cP\subseteq \cP^{(r)}$, determine the order of magnitude of the random variable $z_{\cP}(\RM^{(r)}_{n})$.}

\subsection{Graph matchings ($r=2$)}\label{graphs}

 Earlier results about $z_{\cP}(\RM^{(r)}_{n})$ dealt exclusively with the case $r=2$. We present them jointly below, setting $\cP_1=\{AABB,ABAB\}$, $\cP_2=\{AABB, ABBA\}$ and $\cR=\{ABAB,ABBA\}$, for convenience.

 We say that an event ${\mathcal A}_n$ holds \emph{asymptotically almost surely}, or shortly, \emph{a.a.s.}, if $\Pr({\mathcal A}_n)\to1$, as $n\to\infty$. Also, for a sequence of random variables $X_n$, we write $X_n\sim a_n$ if $X_n/a_n$ converges to 1 in probability.

 \begin{theorem}\label{thm:r=2} We have

\begin{enumerate}
\item[(a)]  $z_{\{ABAB\}}(\RM^{(2)}_{n})\sim\sqrt{2n},\; z_{\{ABBA\}}(\RM^{(2)}_{n})\sim\sqrt{2n}$ (Baik and Rains~\cite{BaikRains}),
\item[(b)]  $z_{\{AABB\}}(\RM^{(2)}_{n})\sim\sqrt{n/\pi}$ (Justicz, Scheinerman and Winkler~\cite{JSW}),
\item[(c)]  A.a.s.\;$z_{\cP_i}(\RM^{(2)}_{n})=\Theta(n^{1/2})\;, i=1,2$ (Dudek, Grytczuk and Ruci\'nski~\cite{DGR-socks}),
\item[(d)]  $z_{\cR}(\RM^{(2)}_{n})\sim \frac n2$ (Scheinerman~\cite{Scheinerman1988}; Dudek, Grytczuk and Ruci\'nski~\cite{DGR-socks}).
\end{enumerate}
\end{theorem}
Let us remark that part (c) was given in  \cite{DGR-socks} without proof, so we prove it in Section~\ref{generic}.

 But first we  emphasize the usefulness of the results in Theorem~\ref{thm:r=2} for various classes of graphs defined on randomly selected intervals.
 Indeed, any ordered matching $M$ of size $n$ can be interpreted as a set of $n$ intervals (identified with the edges of $M$) with no common endpoints. For a subset $\cP\subset\cP^{(2)}$ we may define a graph $G(M,\cP)$ on vertex set $V(G)=M$, where for any $e,f\in M$, we have $\{e,f\}\in E(G)$ if $e$ and $f$ form a pattern belonging to $\cP$. The six non-trivial subsets of $\cP^{(2)}$ give rise to six well studied types of graphs: interval graphs ($\cP=\cR$), overlap graphs, a.k.a. circle graphs ($\cP=\{ABAB\}$), containment graphs, a.k.a. permutation graphs ($\cP=\{ABBA\}$) and their complements. Thus, Theorem~\ref{thm:r=2} provides estimates on their clique and independence numbers. For instance, parts (d) and (b) yield, respectively, estimates on the clique number and the independence number of a random interval graph of size $n$, studied already by Scheinerman in 1988 (\cite{Scheinerman1988}).

\begin{figure}[ht]

\begin{center}
\scalebox{0.7}
{
\begin{tikzpicture}
\begin{axis}[
    xmin=0, xmax=40,
    ymin=0, ymax=40,
    xtick={0,5,15,25,35,40},
    ytick={0,5,15,25,35,40},
    xticklabels={\ ,\underbar{$A$},\underbar{$B$},\underbar{$B$}, \underbar{$A$},\ },
    yticklabels={\ ,$\overline{B}$,$\overline{A}$,$\overline{A}$, $\overline{B}$,\ },
    legend pos=north west,
    ymajorgrids=true,
    xmajorgrids=true,
    grid style=dashed,
]

 \addplot[color=blue, fill=blue, fill opacity=0.5, mark=*,mark size=1.5pt]
 coordinates {
        (5, 15)
        (5, 25)
        (35, 25)
        (35, 15)
        (5, 15)
    };

    \addplot[color=red,  fill=red, fill opacity=0.5, mark=*,mark size=1.5pt]
    coordinates {
        (15, 5)
        (15, 35)
        (25, 35)
        (25, 5)
        (15, 5)
    };
\end{axis}
\end{tikzpicture}
\begin{tikzpicture}
\begin{axis}[
    xmin=0, xmax=40,
    ymin=0, ymax=40,
    xtick={0,5,15,25,35,40},
    ytick={0,5,15,25,35,40},
    xticklabels={\ ,\underbar{$A$},\underbar{$B$},\underbar{$B$}, \underbar{$A$},\ },
    yticklabels={\ ,$\overline{B}$,$\overline{A}$,$\overline{B}$, $\overline{A}$,\ },
    legend pos=north west,
    ymajorgrids=true,
    xmajorgrids=true,
    grid style=dashed,
]

 \addplot[color=blue, fill=blue, fill opacity=0.5, mark=*,mark size=1.5pt]
 coordinates {
        (5, 15)
        (5, 35)
        (35, 35)
        (35, 15)
        (5, 15)
    };

    \addplot[color=red,  fill=red, fill opacity=0.5, mark=*,mark size=1.5pt]
    coordinates {
        (15, 5)
        (15, 25)
        (25, 25)
        (25, 5)
        (15, 5)
    };
\end{axis}
\end{tikzpicture}
}
\end{center}

\caption{Intersecting rectangles defined by \underbar{$ABB$}$\overline{B}$\underbar{$A$}$\overline{AAB}$ and \underbar{$ABB$}$\overline{B}$\underbar{$A$}$\overline{ABA}$.}
\label{fig:rectangles}

\end{figure}
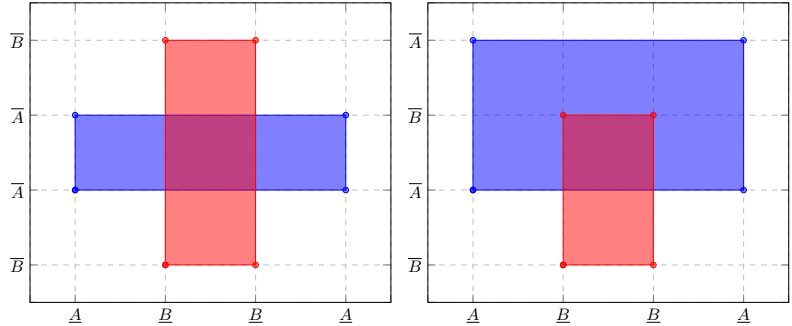

For $r>2$ one can also define graphs $G(M,\cP)$, though geometric interpretations do not come so profusely. A certain special class of $r$-word representable graphs was studied  in~\cite{KL}. In turn, ordered 4-matchings can be used to define families of axis-parallel rectangles (see, e.g.,~\cite{Chud}). Indeed, we may interpret the first two vertices of each edge as an interval on the $x$-axis and the other two as an interval on the $y$-axis, and the two intervals define a rectangle in a usual way. For example, $ABBBAAAB$ is split as \underbar{$ABB$}$\overline{B}$\underbar{$A$}$\overline{AAB}$, where the underlined and overlined symbols are placed on the corresponding axes (see Figure~\ref{fig:rectangles}, left). Such a pair of intersecting rectangles, where no corner of one rectangle lies inside the other, is called \emph{crossing}, and only one more 4-pattern, $ABBABAAB$, yields such a pair. An example of a non-crossing intersecting pair of rectangles is presented in Figure~\ref{fig:rectangles}, right).
Altogether, out of all 35 4-patterns, twelve define intersecting rectangles, and then, denoting their family by $\cP$, the random variable $z_{\cP}(\RM^{(4)}_{n})$ is the size of the largest clique in the intersection graph  of a random set of $n$ axis-parallel rectangles obtained from $\RM^{(4)}_{n}$.

 \subsection{Hyper-matchings}

 The results in parts (a) and (b) of Theorem~\ref{thm:r=2} were extended to arbitrary $r\geqslant 3$ (but in a weaker form) in \cite{JCTB_paper}, with a necessary restriction to \emph{collectable} patterns $P$, that is,  patterns that allow building arbitrarily large $P$-cliques.

\begin{theorem}[Dudek, Grytczuk and Ruci\'nski~\cite{JCTB_paper}]\label{thm:random_one_pattern} For a fixed $r\ge2$ and for every collectable $r$-pattern $P$, a.a.s.
	\[
	z_{\{P\}}(\RM^{(r)}_{n})=\Theta(n^{1/r}).
	\]
\end{theorem}

The following characterization of collectable patterns was established in \cite[Proposition 2.1]{JCTB_paper}: an $r$-pattern $P$ is collectable if and only if it is \emph{splittable} into consecutive \emph{blocks}, $P=S_1\cdots S_s$, of the form $S_i=A^tB^t$ or $S_i=B^tA^t$. Obviously, such a splitting is unique.  For instance,
 \begin{equation}\label{P}
 P=AB|BBAA|BA|AAABBB|BBAA|BA|AB|AB
 \end{equation}
is a collectable $12$-pattern with the suitable splitting into eight blocks $S_i$.
 From this characterization, one can easily deduce that for every $r\geqslant 3$ there are exactly $3^{r-1}$ collectable patterns in $\cP^{(r)}$. It was also proved in \cite[Proposition 2.1]{JCTB_paper} that non-collectable patterns cannot even create $P$-cliques of size three.

Every collectable $r$-pattern $P$ determines a \emph{composition} $\lambda_P=(\lambda_1,\dots,\lambda_s)$ of the integer $r$, that is, an \emph{ordered partition} $r=\lambda_1+\cdots+\lambda_s$, where $\lambda_i=|S_i|/2$, for $i=1,\dots,s$.
  For the 12-pattern $P$ defined in~\eqref{P} above, we have $\lambda_P=(1,2,1,3,2,1,1,1)$.

Note that for a collectable pattern $P$ with $s$ blocks in its splitting there are exactly $2^{s-1}$  patterns $Q$ with $\lambda_Q=\lambda_P$ (including $P$).   Let us call such pairs $\{P,Q\}$ of collectable $r$-patterns \emph{harmonious}.
We denote by $\mathcal P(\lambda)$ the set of all patterns $P$ with composition $\lambda_P=\lambda$. So, all members of $\mathcal P(\lambda)$ are mutually harmonious. For $r=2$, patterns $ABAB$ and $ABBA$ form a (unique)  harmonious pair which seems to be a reason for the value of $z_{\cR}(\RM^{(2)}_{n})$ to a.a.s. be much larger than for the other two pairs (cf. Theorem~\ref{thm:r=2}).

 An important subclass of matchings is that of \emph{$r$-partite} $r$-matchings where every edge takes one vertex from each of the $r$ consecutive blocks of order $n$ (equivalently, these are the matchings with interval chromatic number $r$ -- cf. \cite{Pach-Tardos}).
 For instance, the following $3$-matching is visibly $3$-partite:
\[\color{Red}{ABCDEFG}\color{Green}{ACGEFBD}\color{Blue}{GFEDCBA}\color{black}.\]
 For $n=2$ we obtain the notion of \emph{$r$-partite patterns} which can also be characterized by having composition $\lambda^{(r)}:=(1,1,\dots,1)$ with all $r$ coordinates equal to 1.
Let $\cR^{(r)}:= \cP(\lambda^{(r)})$ be the set of all $r$-partite $r$-patterns. Clearly, $|\cR^{(r)}|=2^{r-1}$.
For instance, the family of 4-patterns defining intersecting rectangles,  mentioned in Section~\ref{graphs}, contains all 8 members of $\cR^{(4)}$.
It follows from these definitions that any pattern occurring in an $r$-partite matching $M$ must be $r$-partite itself. Conversely, every $\cP$-clique with all $\cP\subseteq \cR^{(r)}$  is an $r$-partite matching.
In particular, $\cR=\cR^{(2)}$, which
  leads to a reformulation of  Theorem~\ref{thm:r=2}(d):  \emph{The maximum size of a bipartite sub-matching in $\RM^{(2)}_{n}$ is asymptotic to $n/2$}.

Table \ref{table:relations} contains all patterns in $\cP^{(3)}$, among which $P_1,P_2,\ldots,P_9$ are collectable, while $P_0$ is the only pathological one. Moreover,  $\cR^{(3)}=\{P_6,P_7 P_8,P_9\}$. Clearly, all members of $\cR^{(3)}$ are mutually harmonious, so are
$P_2=AABBBA$ and $P_3=AABBAB$ (with composition $\lambda=(2,1)$) as well as $P_4=ABBBAA$ and $P_5=ABAABB$ (with composition $\lambda=(1,2)$).

\begin{table}[h!]
	\centering
	\begin{tabular}{ |c|c| }
		\hline
		$P_0=AABABB$&$P_5=ABAABB$\\
		\hline
		$P_1=AAABBB$&$P_6=ABBABA$\\
		\hline
		$P_2=AABBBA$&$P_7=ABBAAB$\\
		\hline
		$P_3=AABBAB$&$P_8=ABABBA$\\
		\hline
		$P_4=ABBBAA$&$P_9=ABABAB$\\
			\hline
	\end{tabular}
\caption{All possible 3-patterns.}
\label{table:relations}
\end{table}

\section{New results}
Theorem~\ref{thm:r=2} gives a pretty much  full picture of the order of magnitude of $z_{\cP}(\RM^{(2)}_{n})$, the size of the largest $\cP$-clique in $(\RM^{(2)}_{n})$.
For $r\ge3$, besides Theorem~\ref{thm:random_one_pattern} for $|\cP|=1$, the grasp on $z_{\mathcal P}(\RM_n)$ is still far from complete. In this section we present  some new results in this direction.

One case in which we are able to pinpoint the value of $z_{\mathcal P}(\RM^{(r)}_n)$ quite precisely is when $\cP =\cR^{(r)}$. Generalizing Theorem~\ref{thm:r=2}(d), we show that the largest $\cR^{(r)}$-clique in $\RM^{(r)}_{n}$ has size linear in $n$ and find the value of the multiplicative constant asymptotically.

\begin{theorem}\label{thm:random_r-partite}
For $r\ge2$, let $\cR^{(r)}$ be the set of all $r$-partite $r$-patterns. Then  a.a.s. \[z_{\cR^{(r)}}(\RM^{(r)}_{n})\sim\frac{(r-1)!}{r^{r-1}}n.\]
Equivalently,  $R_n^{(r)}$,  the largest (lexicographically first) $r$-partite sub-matching of $\RM^{(r)}_{n}$ has a.a.s. size
$
|R_n^{(r)}|\sim\frac{(r-1)!}{r^{r-1}}n.
$
\end{theorem}
It follows that the size of the largest clique in the intersection graph  of a random set of $n$ axis-parallel rectangles obtained from $\RM^{(4)}_{n}$ (as defined in Section~\ref{graphs}) is a.a.s. $\Theta(n)$.

 Our next result determines, for all $r\ge2$,  the order of magnitude of $z_{\mathcal P}(\RM^{(r)}_{n})$ for  \emph{all} sets of $r$-patterns $\mathcal{P}$ with $|\cP|=2$.
  We call a pair of  $r$-patterns $\{P,Q\}$ a \emph{mismatch} if one of them is collectable, while the other is not, or both are collectable, but then not harmonious.

\begin{theorem}\label{thm:random_two_patterns_gen}
For $r\ge2$, let $\cP=\{P,Q\}$ be a pair of distinct $r$-patterns. Then a.a.s.
\[
z_{\mathcal P}(\RM^{(r)}_{n})=\begin{cases}\Theta(n^{1/(r-1)})&\text{if $P$ and $Q$ are harmonious,}\\
	\Theta(n^{1/r})&\text{if $P$ and $Q$ form a mismatch,}\\
z&\text{if neither $P$ nor $Q$ is collectable,}
\end{cases}
\]
where $z=z_{\cP}$ is an integer, $2\le z\le 5$.
\end{theorem}
\noindent It is worth noticing that the value of $z_{\mathcal P}(\RM^{(r)}_{n})$ in the above theorem does not depend on $\lambda_P$ and $\lambda_Q$ but only on $r$ and the structural relation between $P$ and $Q$. As we will demonstrate in the proof, the constant $z$ appearing in the last statement equals the largest size of a $\{P,Q\}$-clique whatsoever.

As an application, taking $P=ABBBAAAB$ and $Q=ABBABAAB$, we infer that the largest subset of mutually crossing rectangles among $n$ axis-parallel random rectangles obtained from $\RM^{(4)}_{n}$ (as described in Section~\ref{graphs})  a.a.s. has size $\Theta(n^{1/4})$.

\medskip

For \emph{triples} of  $r$-patterns we are far from a complete picture. For $r=3$, we  only know the answer when all three 3-patterns are tripartite.
 In fact, the proof can be carried over for any $r\ge3$ (and $2^{r-1}-1$ $r$-partite $r$-patterns), so we state the result in its general form.

\begin{proposition}\label{thm:random_three_patterns_stronger}
For $r\ge2$, let $\mathcal{P}\subset\cR^{(r)}$ with $|\cP|=2^{r-1}-1$. Then, a.a.s.\ $z_{\mathcal P}(\RM^{(r)}_{n})=\Theta(n^{1-1/r})$.
\end{proposition}

Most proofs, along with essential tools and preliminary observations, are presented in the next two sections. In particular, we will state an Azuma-Hoeffding-type concentration inequality for random permutations and use it to prove  Theorem~\ref{thm:random_r-partite}.    For other proofs we will use a variety of tools and techniques, including the standard first moment method, Dilworth's Lemma, estimates on pattern avoiding permutations,  and  Dyck words.  In the Appendix, we present the remaining proofs of Theorem~\ref{thm:random_two_patterns_gen}: the non-collectable case and the harmonious case (the lower bound), the latter based on a Talagrand-type strengthening of Azuma-Hoeffding.

\subsection{Open problems} Here we gather a few open problems we would like to work on in the future.

{\bf 1.} For an ordered set of $rn$ vertices, imagine a random process in which, sequentially, for each  next available vertex $v$  a set of $r-1$ vertices is selected uniformly at random to form with $v$ an $r$-element edge.  For a given set $\cP$ of $r$-patterns,   determine the order of magnitude of the size of the largest $\mathcal P$-clique after $t$ steps of this process, for any given $t=t(n)$. Reversely, given a function $k=k(n)$, determine the threshold $t(n)$, or even the hitting time, for the process to grow a $\mathcal P$-clique of size $k(n)$.

More realistically, though, we hope to be able to answer some further questions concerning the final stage of the process ($t=n$), that is, the random ordered $r$-matching $\RM^{(r)}_n$.

{\bf 2.} The first two questions relate to the limitations of the disturbing role  of non-collectable patterns.
A matching is \emph{clean} if every pair of edges forms a collectable pattern. What is the size of the largest clean sub-matching of the random $r$-matching  $\RM^{(r)}_n$? Theorem~\ref{thm:random_r-partite} provides a lower bound here, so the answer to this question should be linear in $n$ with the coefficient at least as large as $(r-1)!r^{1-r}$.
A more general question can be stated as follows.
Is it true that for every set $\mathcal P$ of collectable $r$-patterns and every set $\mathcal Q$ of non-collectable $r$-patterns, a.a.s.\
$z_{\mathcal P}(\RM^{(r)}_{n})=\Theta(z_{\mathcal P\cup\mathcal Q}(\RM^{(r)}_{n}))$ ?

{\bf 3.} With respect to Theorem~\ref{thm:random_two_patterns_gen}, we would like to extend it to \emph{all triplets} of $r$-patterns. For mutually harmonious triplets, so far we only have the result of Proposition~\ref{thm:random_three_patterns_stronger} for $r=3$: for any $3$-element subset  $\cP\subset\cR^{(3)}=\{P_6,P_7,P_8,P_9\}$,  a.a.s.\ $z_{\mathcal P}(\RM^{(r)}_{n})=\Theta(n^{2/3})$. We believe that for all other triplets of collectable 3-patterns the value of $z_{\mathcal P}(\RM^{(r)}_{n})$ should  be lower than that.
In fact, we conjecture that for every $r\ge3$, if $\cP$ consists of three pairwise non-harmonious $r$-patterns (at least one collectable), then a.a.s.\ $z_{\mathcal P}(\RM^{(r)}_{n})=\Theta(n^{1/r})$,  the same as  for single collectable patterns and mismatch pairs.

\section{Probabilistic tools} 
In this  section we present a couple of probabilistic tools and apply them right away to prove some of our results.

\subsection{ Concentration inequality for random matchings}\label{section:random}

Recall that $\RM^{(r)}_{n}$ denotes a random ordered $r$-uniform matching of size $n$, that is, a~matching picked uniformly at random out of the set $\cM_n^{(r)}$ of all
\[
\alpha^{(r)}_n:=\frac{(rn)!}{(r!)^n\, n!}
\]
ordered $r$-matchings on the set $[rn]$. For future reference note that the probability that a given $r$-element subset $e$ of $[rn]$ is an edge of $\RM^{(r)}_{n}$ is
\begin{equation}\label{1edge}
\PP(e\in \RM^{(r)}_{n})=\frac{\alpha^{(r)}_{n-1}}{\alpha^{(r)}_n}=\frac1{\binom{rn-1}{r-1}}.
\end{equation}

There is another, equivalent way of drawing $\RM^{(r)}_{n}$, the \emph{permutational scheme}.
Indeed, let $\pi$ be a permutation of $[rn]$. It can be chopped off into an $r$-matching \[M_\pi:=\left\{\{\pi(1),\dots,\pi(r)\}, \{\pi(r+1),\dots,\pi(2r)\},\dots,\{\pi(rn-r+1),\dots,\pi(rn)\}\right\}\] and, clearly, there are exactly $(r!)^n n!$ permutations $\pi$ yielding the same matching.
This means that the probability of choosing a matching with a given property, while picking it uniformly at random, is the same as when  picking a random permutation and then extracting from it an $r$-matching in the above prescribed way.

We will benefit from this straightforward observation, as it allows direct applications of
high concentration results for random permutations in the context of ordered matchings.
 One of them is the Azuma-Hoeffding inequality for random permutations (see, e.g., Lemma 11 in~\cite{FP} or  Section 3.2 in~\cite{McDiarmid98}).
 By swapping two elements in a permutation $\pi_1$ we mean fixing two indices $i<j$ and creating a new permutation $\pi_2$ with $\pi_2(i)=\pi_1(j)$, $\pi_2(j)=\pi_1(i)$, and $\pi_2(\ell)=\pi_1(\ell)$ for all $\ell\neq i,j$. Let $\Pi_{N}$ denote a permutation selected uniformly at random from all $N!$ permutations of $[N]$.

\begin{lemma}[Frieze and Pittel~\cite{FP} \textnormal{or} McDiarmid~\cite{McDiarmid98}]\label{azuma}
 Let $h(\pi)$ be a function defined on the set of all permutations of order $N$ which satisfies the following Lipschitz-type condition: there exists a constant $c>0$ such that whenever a permutation $\pi_2$ is obtained from a permutation $\pi_1$ by swapping two elements, we have $|h(\pi_1)-h(\pi_2)|\le c$.
Then, for every $\eta>0$,
\[
\PP(|h(\Pi_{N})-\E[h(\Pi_{N})]|\ge \eta)\le 2\exp\{-2\eta^2/(c^2N)\}.
\]
\end{lemma}

We will now employ Lemma~\ref{azuma} to prove Theorem~\ref{thm:random_r-partite}. 

\begin{proof}[Theorem~\ref{thm:random_r-partite}]
For given disjoint subsets $A_1,\dots,A_r \subseteq [rn]$ we say that an $r$-edge $e$ \emph{spans} these sets if $|e\cap A_i|=1$ for each $1\le i\le r$ (see Figure\ \ref{span}).

\begin{figure}[ht]
\captionsetup[subfigure]{labelformat=empty}
\begin{center}

\scalebox{0.9}
{
\centering
\begin{tikzpicture}
[line width = .5pt,
vtx/.style={circle,draw,black,very thick,fill=black, line width = 1pt, inner sep=0pt},
]

    \node[vtx] (1) at (0,0) {};
    \node[vtx] (2) at (2,0) {};
    \node[vtx] (3) at (4,0) {};
    \coordinate (34) at (6,1) {};
    \node[vtx] (4) at (8,0) {};
    \node[vtx] (5) at (10,0) {};

    \draw[line width=0.3mm, color=lightgray]  (1) -- (5);

    \draw[line width=2.5mm, color=gray, line cap=round]  (1.5,0) -- (2.5,0);
    \draw[line width=2.5mm, color=gray, line cap=round]  (3.75,0) -- (4.75,0);
    \draw[line width=2.5mm, color=gray, line cap=round]  (7.25,0) -- (8.25,0);

    \node[color=gray] at (2.6,-0.35) {$A_1$};
    \node[color=gray] at (4.85,-0.35) {$A_2$};
    \node[color=gray] at (8.35,-0.35) {$A_3$};
    \node[color=blue] at (4,0.7) {$e$};

    \draw[line width=0.5mm, color=blue, outer sep=2mm] plot [smooth, tension=2] coordinates {(3) (34) (4)};
    \draw[line width=0.5mm, color=blue, outer sep=2mm] (3) arc (0:180:1);

    \fill[fill=black, outer sep=1mm]   (1) circle (0.1) node [below] {$1$};
    \fill[fill=black, outer sep=1mm]   (2) circle (0.1);
    \fill[fill=black, outer sep=1mm]  (3) circle (0.1);
    \fill[fill=black, outer sep=1mm]   (4) circle (0.1);
    \fill[fill=black, outer sep=1mm]   (5) circle (0.1) node [below] {$3n$};

\end{tikzpicture}
}

\end{center}

\caption{A 3-edge spanning three sets.}
\label{span}
\end{figure}
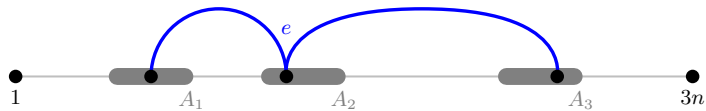

 Recall that $\cR^{(r)}$ stands for the set of all $r$-partite $r$-patterns and let $M$ be an $\cR^{(r)}$-clique on vertex set $[rn]$. Since $M$ is  $r$-partite itself, the vertices of $M$ can be partitioned into $r$ sets $A_i \subseteq [rn]$ with $|A_i|=|M|$,  $1\le i\le r$, and such that for each $i=1,\dots,r-1$, all elements of $A_i$ appear before elements of $A_{i+1}$, and each edge of $M$ spans the sets $A_1,\dots,A_r$. Clearly, there is a partition $[rn]=B_1\cup\cdots\cup B_r$ into pairwise disjoint blocks such that for each $i=1,\dots,r$ we have $A_i\subseteq B_i$.
 Consequently, in estimating the largest size of an $r$-partite sub-matching of $\RM^{(r)}_{n}$ it suffices to consider only matchings whose edges span all sets $B_1,\dots,B_r$ of such a partition.

For a given partition $\mathbf{B}$ of $[rn]$ into blocks $B_1,\dots,B_r$ define a random variable $X_{\mathbf{B}}$ that counts the number of edges of  $\RM^{(r)}_{n}$ spanning the sets $B_i$. Note that the number of choices for $\mathbf{B}$ is $O(n^{r-1})$.
 We will show that a.a.s.\ for every $\mathbf{B}$ we have $X_{\mathbf{B}} \le \frac{(r-1)!}{r^{r-1}}n(1+o(1))$ and for some $\mathbf{B}$ this bound is asymptotically achieved.

To this end, fix a partition $\mathbf{B}$ with $|B_i|=b_i$, $i=1,\dots,r$, where $\sum_{i=1}^rb_i=rn$. For each $r$-element set $e$ which spans the sets $B_1,\dots,B_r$, let $I_e=1$ if $e\in\RM^{(r)}_{n}$ and $I_e=0$ otherwise.
Then $X_{\mathbf{B}}=\sum_eI_e$ and  by~\eqref{1edge} and the linearity of expectation
\[
\E(X_{\mathbf{B}}) =\sum_e\E(I_e)
=   \frac{\prod_{i=1}^rb_i}{\binom{rn-1}{r-1}}.
\]
Note that for the equipartition $\mathbf{B^{*}}$ with all $b_i=n$, we get
\begin{equation}\label{B*}
\E(X_{\mathbf{B^{*}}})=  \frac{n^r}{\binom{rn-1}{r-1}} \sim \frac{(r-1)!}{r^{r-1}}n.
\end{equation}

It turns out that for all $\mathbf{B}$, the above quantity is an upper bound on $\E(X_{\mathbf{B}})$.
Indeed, by the inequality of arithmetic and geometric means, we get
\[
\prod_{i=1}^r b_i
\le \left( \frac{\sum_{i=1}^r b_i}{r} \right)^r
= n^r.
\]
Thus, for all $\mathbf{B}$,
\begin{equation}\label{ineq:X_B}
\E(X_{\mathbf{B}}) \le  \frac{n^r}{\binom{rn-1}{r-1}} \le (1+o(1))\frac{(r-1)!}{r^{r-1}}n.
\end{equation}

Now we are going to apply Lemma~\ref{azuma} to $X_{\mathbf{B}}$ for each $\mathbf{B}$ separately.
With respect to the Lipschitz condition the following observation is crucial for us.
Let $\pi_1$ be a permutation of $[rn]$ and let permutation $\pi_2$ be obtained from $\pi_1$ by swapping two elements. This way we can destroy (or create) at most two edges which may jointly contribute at most $2$ to $X_{\mathbf{B}}$. Hence, each random variable $X_{\mathbf{B}}$ satisfies the Lipschitz condition with $c=2$. Thus, Lemma~\ref{azuma} applied with $N=rn$, $c=2$, and $\eta=\sqrt{\omega n\log n}$ (where $\omega:=\omega(n)\to\infty$) yields that
$$\PP\left(|X_{\mathbf{B}}-\E(X_{\mathbf{B}})| \ge \sqrt{\omega n\log n}\right)\le 2n^{-\omega/2r}.$$
 This, together with~\eqref{ineq:X_B} and the union bound over all partitions $\mathbf{B}$, implies that a.a.s.\ the size of the largest $\cR^{(r)}$-clique, $\max_{\mathbf{B}}X_{\mathbf{B}}$, is at most $(1+o(1))\frac{(r-1)!}{r^{r-1}}n$. Moreover, by the same token, in view of~\eqref{B*}, this upper bound is asymptotically achieved by $X_{\mathbf{B^{*}}}$, which completes the proof of~Theorem~\ref{thm:random_r-partite}. \qed
 \end{proof}

\subsection{Generic first moment method}\label{generic}

For most  upper bound proofs in this paper we use a standard first moment method. Here we state a generic, technical ``meta-statement'' which will be applied throughout. For a set of $r$-patterns $\cP$, let $a_{\mathcal P}(k)$ denote the  number of $\mathcal P$-cliques among the matchings in $\cM_k^{(r)}$, that is, the  number of distinct $\mathcal P$-cliques one can build on a given ordered set of $rk$ vertices.

\begin{lemma}\label{meta} For $r\ge2$, let $\cP$ be a set of $r$-patterns.
If there are  constants $C>0$ and  $0\le x<r$ such that  $a_{\mathcal P}(k)\le C^{k}k^{kx}$ for all $k\ge2$, then a.a.s.\ $z_{\cP}(\RM^{(r)}_{n})=O(n^{1/(r-x)})$.
\end{lemma}

\begin{proof}
For each $k\ge2$, let $X_k$ be the number of $\cP$-cliques of size $k$ in $\RM^{(r)}_n$.
In order to compute the expectation of $X_k$, one has to first choose a set $S$ of $rk$ vertices (out of all $rn$ vertices) on which a $\cP$-clique will be planted. Formally, for each $S\in\binom{[rn]}{rk}$ define an indicator random variable $I_S=1$ if there is a $\cP$-clique on $S$ in $\RM^{(r)}_n$, and $I_S=0$ otherwise.
Thus,
\[
\E I_S = \Pr(I_S=1)= a_{\mathcal P}(k)\cdot \frac{\alpha^{(r)}_{n-k}}{\alpha^{(r)}_n}
\]
and by the linearity of expectation, noticing that $X_k=\sum_S I_S$,
$$\E X_k=\binom{rn}{rk}a_{\cP}(k)\frac{\alpha^{(r)}_{n-k}}{\alpha^{(r)}_n}=\frac{(r!)^k}{(rk)!} \cdot a_{\cP}(k) \cdot \frac{n!}{(n-k)!}\le
\left( \frac{C' n}{k^{r-x}} \right)^k,
$$
where $C'=Cr!e^r r^{-r}$.
Let $k_0=\lceil C''n^{1/(r-x)}\rceil = \Theta(n^{1/(r-x)})$, where $C''>(C')^{1/(r-x)}$.
  Then, by  Markov's inequality,
\[
\Pr( X_{k_0}\ge1)
\le  \left( \frac{C'n}{(k_0)^{r-x}}\right)^{k_0}
= o(1).
\]
Finally, note that  $ X_{k_0}\ge1$ is equivalent to $z_{\cP}(\RM^{(r)}_{n})\ge k_0$. \qed
\end{proof}

\medskip

As a first application, we now provide a quick proof of  Theorem~\ref{thm:r=2}(c).

 \begin{proof}[Theorem~\ref{thm:r=2}(c)]
 It suffices to prove the upper bound only, since the lower bound follows by Theorem~\ref{thm:random_one_pattern} and  monotonicity.
First note that a matching is a $\cP_1$-clique if and only if it is $ABBA$-free, or nesting-free, while it is  a $\cP_2$-clique if and only if it is $ABAB$-free, or crossing-free. Both these families of matchings (with $k$ edges) have been enumerated by Stanley~\cite{StanleyCatalan} who showed that their cardinalities equal the Catalan number $C_k=\tfrac1{k+1}\binom{2k}k$. Since $C_k<4^k$, the result follows by Lemma~\ref{meta} with $C=4$ and $x=0$. \qed
\end{proof}

 As another application of Lemma~\ref{meta}, let us prove the upper bound in Theorem~\ref{thm:random_two_patterns_gen} in the case of harmonious pairs.
In doing so, we will be relying on the following operation of ``gluing'' disjoint matchings of equal size. Let $V$ be an ordered set and let $M_1=\{e_1,\dots,e_k\}$ be an $r_1$-matching on vertex set $V_1$ and  $M_2=\{f_1,\dots,f_k\}$ be an $r_2$-matching on  vertex set $V_2$, with $V_1\cup V_2\subseteq V$ and $V_1\cap V_2=\emptyset$ (the edges of the matchings are numbered in the order of their left-ends). For a permutation $\sigma:[k]\to[k]$  define the \emph{$\sigma$-product} $M_1\overset{\sigma}{\times} M_2$ of $M_1$ and $M_2$ as the $(r_1+r_2)$-matching with the edge set $\{e_1\cup f_{\sigma(1)},\dots, e_k\cup f_{\sigma(k)}\}$. This definition can be extended naturally to $(\sigma_1,\dots,\sigma_h)$-products of $h+1$ matchings
$M_1\overset{\sigma_1}{\times} M_2\overset{\sigma_2}{\times}\cdots \overset{\sigma_h}{\times} M_{h+1}$.

\begin{proof}[Theorem~\ref{thm:random_two_patterns_gen}, upper bound for harmonious pairs]

Let $\lambda_P=(\lambda_1,\dots,\lambda_s)$ be a composition $r=\lambda_1+\cdots +\lambda_s$ and let $P=S_1S_2\cdots S_s$ and $Q=S_1'S_2'\cdots S_s'$ be a pair of harmonious $r$-patterns with $\lambda_P=\lambda_Q=\lambda$, where $S_i$ and $S_i'$ are the blocks of the splitting with  $2\lambda_i=|S_i|=|S'_i|$, $i=1,\dots,s$.

Each block $S_i$ and $S_i'$ is of the form $(A)^{\lambda_i}(B)^{\lambda_i}$ or $(B)^{\lambda_i}(A)^{\lambda_i}$, the former option mandatory for $i=1$.

We are going to show that $a_{\{P,Q\}}(k)=k!$ which, by Lemma~\ref{meta} with $C=1$ and $x=1$, will imply that $z_{\{P,Q\}}(\RM^{(r)}_{n})= O\left(n^{1/(r-1}\right)$.
Let $T_0$ be the set of indices $i$ for which $S_i=S_i'$, $T_1=[s]\setminus T_0$, and let
$P_j$ be the concatenation of segments $S_i$, with $i\in T_j$, $j=0,1$. Thus, $P_0$ collects those blocks which are identical in $P$ and $Q$, while $P_1$ -- those which are pairwise different.

 Further, let $[rk]=V_1\cup\cdots\cup V_s$ be a partition of $[rk]$ into consecutive blocks, where $|V_i|=k\lambda_i$, for $i=1,\dots,s$. Finally, let $W_j:=\bigcup_{i\in T_j} V_i$, for $j=0,1$.

For every $\{P,Q\}$-clique $K$ on $[rk]$ and for each $j=0,1$, every edge $e\in K$ satisfies $|e\cap W_j|=\sum_{i\in T_j}\lambda_i=:\ell_j$.
For $j=0,1$, let $K[W_j]$ be the ordered $\ell_j$-matching on $W_j$ consisting of all edges $e\cap W_j$ with $e\in K$. Then, clearly, each $K[W_j]$, $j=0,1$, forms the unique $P_j$-clique planted on $W_j$,
and $K$ is of the form
\begin{equation}\label{KWi}
K[W_0]\overset{\sigma}{\times} K[W_1]
\end{equation}
for some  permutation $\sigma$ of $[k]$.

On the other hand, every  permutation $\sigma$ of $[k]$ defines a distinct  $\{P,Q\}$-clique given by~(\ref{KWi}).
We conclude that indeed $a_\cP(k)=k!$. \qed
\end{proof}

\section{Combinatorial tools}
In this  section we present several combinatorial tools and apply them right away to prove some of our results.

\subsection{Posets on matchings}\label{pos} It turns out, as remarked in \cite{GP2019}, that one can often define a partially ordered set (poset) on the set of edges of a random matching $\RM_n^{(r)}$ and then utilize Dilworth's Lemma, which we state here in a version most suited for us.

 \begin{lemma}[Dilworth, 1950]\label{DL}
 If the largest chain in a finite poset $X$ has size at most $a$, then the largest anti-chain in $X$ has size at least $|X|/a$.
 \end{lemma}

A general template of how to proceed in such scenarios is described in the next lemma.

\begin{lemma}\label{template} Let $\cP\subseteq\cP^{(r)}$ be a set of $r$-patterns and let $N$ be a $\cP$-clique in an $r$-matching~$M$.
Suppose  that for a subset $\cQ\subset \cP$  a poset can be defined on $N$  in such a way that for all $e,f\in N$ we have $e<f$ if $\min e<\min f$ and  $e,f$ form an $r$-pattern $Q$ for some  $Q\in\cQ$.
Then,  $z_{\cQ^c}(M)\ge |N|/z_{\cQ}(M)$, where $\cQ^c=\cP\setminus\cQ$.
\end{lemma}

\begin{proof} Note that for such a poset, every chain is a $\cQ$-clique, while every anti-chain is a $\cQ^c$-clique. As $N\subset M$, by monotonicity, $z_{\cQ}(N)\le z_{\cQ}(M)$ and $z_{\cQ^c}(N)\le z_{\cQ^c}(M)$.  By Lemma~\ref{DL} applied to $N$ we thus conclude that
\[
z_{\cQ^c}(M)\ge z_{\cQ^c}(N)\ge |N|/z_{\cQ}(N)\ge |N|/z_{\cQ}(M). \eqno\qed
\]
\end{proof}

  In this paper we benefit from Lemma~\ref{template} only in the case of $r$-partite patterns. We now deduce the lower bound in Proposition~\ref{thm:random_three_patterns_stronger} from the upper bound in Theorem~\ref{thm:random_one_pattern}.

  \begin{proof}[Proposition~\ref{thm:random_three_patterns_stronger}, lower bound]
  Recall that $R^{(r)}_n$ is the largest (lexicographically first) $r$-partite sub-matching of $\RM^{(r)}_n$.
   Let $P=S_1\cdots S_r$, where $|S_i|=2$, $i=1,\dots,r$, be the only $r$-partite pattern not belonging to $\cP$. Further, let $I_1\subseteq [r]$ consist of all those indices $i$ for which $S_i=AB$, and let $I_2=[r]\setminus I_1$.  Given two edges $e,f\in R^{(r)}_n$, let
 \begin{equation*}
e=(x_1<\cdots <x_r)<f=(y_1<\cdots <y_r)\quad\mbox{if}\quad\begin{cases} x_i<y_i\quad\mbox{for all}\quad i\in I_1
\\x_i>y_i\quad\mbox{for all}\quad i\in I_2.\end{cases}
\end{equation*}
In this poset, a chain is just a $P$-clique, while an anti-chain is an $\left(\cR^{(r)}\setminus\{P\}\right)$-clique.   Since by Theorem~\ref{thm:random_one_pattern} a.a.s.\ $z_{P}(\RM^{(r)}_{n})=\Theta(n^{1/r})$, we are in position to apply Lemma~\ref{template} with $\cQ:=\{P\}$, $M:=\RM^{(r)}_{n}$ and $N:=R^{(r)}_n$, and conclude that a.a.s.\
$z_{\cP}(\RM^{(r)}_{n})=\Omega(n^{(r-1)/r})$.~\qed
\end{proof}

\subsection{Pattern avoiding permutations}

For the  proof of the upper bound in Proposition~\ref{thm:random_three_patterns_stronger} we will need a result of Gunby and P\'{a}lv\"{o}lgyi from \cite{GP2019} concerning parallel pattern avoidance by permutations.

For integers $k\ge m$, let $\sigma$ be a permutation of $[k]$ and let $\tau$ be a permutation of $[m]$. We say that $\sigma$ \emph{avoids} $\tau$ if there is no sequence $a_1<\dots <a_m$ such that the sub-permutation  $\sigma(a_1),\dots, \sigma(a_m)$ is order-isomorphic to permutation $\tau$. This notion can be generalized in a natural way to sequences of permutations. Let $\sigma=(\sigma_1,\dots,\sigma_d)$ be a $d$-tuple of permutations of $[k]$ and let $\tau=(\tau_1,\dots,\tau_d)$ be a $d$-tuple of permutations of $[m]$. We say that $\sigma$ \emph{avoids} $\tau$ if there is no sequence $a_1<\cdots <a_m$ such that all $d$  sub-permutations $\sigma_i(a_1),\dots, \sigma_i(a_m)$ are order-isomorphic to corresponding permutations $\tau_i$ for $i=1,\dots,d$.

Generalizing the celebrated result by Marcus and Tardos \cite{MT}, Gunby and P\'{a}lv\"{o}lgyi proved the following estimates.
\begin{theorem}[Gunby and P\'{a}lv\"{o}lgyi~\cite{GP2019}]\label{thm Gunby-Palvolgyi}
	Let $\tau$ be a fixed $d$-tuple of permutations of $[m]$. Let $f_d(k)$ denote the number of $d$-tuples of permutations of $[k]$ avoiding $\tau$. Then there exist positive constants $c_1$ and $c_2$ (depending on $\tau$) such that
	\[
	c_1^k\cdot k^{k(d-1/d)}\leqslant f_d(k)\leqslant c_2^k\cdot k^{k(d-1/d)}.
	\]
\end{theorem}

  \begin{proof}[Proposition~\ref{thm:random_three_patterns_stronger}, upper bound]
  Let $P=S_1\cdots S_r$, where $|S_i|=2$, $i=1,\dots,r$. Further, let $I_1\subseteq [r]$ consist of all those indices $i$ for which $S_i=AB$.
  Let $\cP=\cR^{(r)}\setminus\{P\}$ consist of all $r$-partite patterns but $P$.
  Further, let $[rk]=V_1\cup\cdots\cup V_r$ be a partition of $[rk]$ into consecutive blocks, where $|V_i|=k$, for $i=1,\dots,s$.
  If we were after $\cR^{(r)}$-cliques, we would just take a $(\sigma_1,\dots,\sigma_{r-1})$-product
 $$K[V_1]\overset{\sigma_1}{\times}K[V_2]\overset{\sigma_2}{\times}\cdots\overset{\sigma_{r-1}}{\times} K[V_r].$$
  However, to avoid the forbidden pattern $P$,  the $(r-1)$-tuple of permutations $(\sigma_1,\dots,\sigma_{r-1})$ should collectively avoid the $(r-1)$-tuple of permutations\linebreak $(\tau_1,\dots,\tau_{r-1})$ of $[12]$ where $\tau_i=12$ if $i+1\in I_1$ and $\tau_i=21$ otherwise, $i=1,\dots,r-1$.

Thus, by Theorem~\ref{thm Gunby-Palvolgyi} with $d=r-1$, we have $a_{\cP}(k)\le c_2^kk^{(r-1-1/(r-1))k}$. Consequently, by  Lemma~\ref{meta} with $C=c_2$ and $x=r-1-1/(r-1)$, we get $z_{\cP}(\RM^{(r)}_{n})= O\left(n^{1-1/r}\right)$, which completes the proof. \qed
\end{proof}

Let us mention that the same bound on $a_{\cP}(k)$ was proved in \cite[Theorem 1.13]{AJKS}  by a different method (not relying on
Theorem~\ref{thm Gunby-Palvolgyi}).

\subsection{Traces of matchings}\label{trace}
In some proofs of upper bounds, in order to estimate the number $a_{\cP}(k)$ of distinct $\cP$-cliques of size $k$ appearing in Lemma~\ref{meta}, we employ the concept of the trace of a matching, introduced in \cite{DGR-socks}. 

Let $M$ be an ordered $r$-matching of size $n$. We  assign number $i$ to the $i$-th (from the left) vertex of each edge of $M$. The obtained $r$-ary sequence (with alphabet $[r]$) will be called the \emph{trace of $M$} and denoted by $\tr(M)$. For instance,
for $M=AABCBDBDACCD$, we have $\tr(M)=121121323233$. In general, there may be many matchings with the same  trace, e.g.,  $M'=AABCCDADDBBC$ has $\tr(M')=\tr(M)$.
 However, for some restricted families of matchings $\cF$ it may happen that no two members of $\cF$ have the same trace. Then, knowing the trace of a matching $M\in\cF$, one is able to uniquely reconstruct $M$. In that case we call the family $\cF$ \emph{reconstructible}.
Using this concept, we obtain the following upper bound on the size of the largest $\cP$-clique in $\RM_n^{(r)}$ for a class of sets $\cP$.


\begin{observation}\label{recon}
Let $\cP$ be a set of $r$-patterns such that the family of all $\cP$-cliques is reconstructible. Then $z_{\cP}(\RM_n^{(r)})=O(n^{1/r})$.
\end{observation}
\begin{proof}
By assumption,  there are no more $\cP$-cliques on a vertex set of size $rk$ than there are words of length $rk$ over alphabet $[r]$. So, $a_{\cP}(k)\leqslant r^{rk}$ and the conclusion follows by Lemma \ref{meta} with $x=0$ and $C=r^r$. \qed
\end{proof}

\begin{remark} In fact, the traces of ordered matchings have a special structure: for every pair $i<j$, in every prefix of the trace $\tr(M)$ the number of occurrences of $i$ is not smaller than the number of occurrences of $j$. For instance, for $r=2$ we obtain the well known family of $\emph{Dyck words}$ enumerated by \emph{Catalan numbers}. For arbitrary $r\geqslant 2$ one gets families of \emph{$r$-dimensional Dyck words} enumerated jointly in  table A060854 of the OEIS database (cf.\ \cite{MMoo}).
  Due to the robustness of our estimates, this is, however, irrelevant for us.
\end{remark}

As an application of Observation~\ref{recon}, we will prove  Theorem~\ref{thm:random_two_patterns_gen} for mismatched pairs.
But  first we make an observation about the inheritance  of the property of being mismatched.
For an $r$-pattern $P$ and $1\le j\le r$, let $P^{-j}$ denote the $(r-1)$-pattern obtained  by removing the $j$-th letter $A$ and the $j$-th letter $B$ from $P$. For instance, if $P=AABBBABA$, then $P^{-1}=P^{-2}=ABBABA$, while $P^{-3}=AABBBA$. Clearly, if $P$ is collectable, then so is $P^{-j}$ for every $j$, but the opposite does not necessarily hold.

\begin{observation}\label{obs:1}
For every $r\geq 3$ and any mismatched pair $\{P,Q\}$, there exists $1\le j\le r$ such that $\{P^{-j},Q^{-j}\}$ is a mismatch, too.
\end{observation}

\begin{proof}[Theorem~\ref{thm:random_two_patterns_gen}, mismatched pairs]
It suffices to prove the upper bound only, since the lower bound follows by Theorem~\ref{thm:random_one_pattern} and  monotonicity.

Let $\{P,Q\}$ be a mismatched pair of $r$-patterns.
In view of Observation~\ref{recon}, it is enough to demonstrate that the family of $\{P,Q\}$-cliques is reconstructible.
We proceed by induction on $r$.

For $r=2$, recall that there are just two mismatched pairs,\linebreak  $\mathcal P_1=\{AABB,\;ABAB\}$ and $\mathcal P_2=\{AABB,\;ABBA\}$. The bijections between nesting-free and crossing-free matchings on one side and the Dyck binary words on the other, used by Stanley \cite{StanleyCatalan} to enumerate both families of matchings (see the proof of Theorem~\ref{thm:r=2}(c) in Section~\ref{generic}), show that both, $\mathcal P_1$-cliques and $\mathcal P_2$-cliques, are indeed reconstructible.

Assume now that $r\geq 3$ and that we have proved the reconstructibility for $r-1$.
 Let $M$ and $N$ be two different $\{P,Q\}$-cliques of size $k$ on vertex set~$V$. We are going to show that $\tr(M)\neq \tr(N)$. To this end, let  $j$ be given by Observation~\ref{obs:1}. After removing the $j$-th letter from each edge of $M$ and $N$, we obtain $\{P^{-j},Q^{-j}\}$-cliques $M^{-j}$ and $N^{-j}$, where  the pair $\{P^{-j},Q^{-j}\}$ is still a mismatch.
 If $V(M^{-j})\neq V(N^{-j})$, then $\tr(M)\neq \tr(N)$ as the digit $j$ appears in $\tr(M)$ and $\tr(N)$ on different positions.

 Assume then that
 $V(M^{-j})= V(N^{-j})=V\setminus\{v_1,\dots,v_k\}$. We are going to show that $M^{-j}\neq N^{-j}$ and then apply induction. Suppose that $M^{-j}= N^{-j}$.
  Let $e_1,\dots, e_k$ be the edges of $M$ and $e'_1,\dots, e'_k$ be the corresponding edges of $M^{-j}$. We may assume that $e_i=e_i'\cup\{v_i\}$, $i=1,\dots,k$. Then $N=\{e_i'\cup\{v_{\alpha(i)}\}: i=1,\dots,k\}$ where $\alpha$ is a permutation of $[k]$. Since $M\neq N$, $\alpha$ is not an identity and so there is an inversion in $\alpha$: a pair of indices $i<\ell$ such that $\alpha(i)>\alpha(\ell)$. W.l.o.g, let the pair $e_i,e_\ell$ form the $r$-pattern $P$, so $e_i',e_\ell'$ form $P^{-j}$. However, due to the inversion, the edges $e_i'\cup\{v_{\alpha(i)}\}$ and $e_\ell'\cup\{v_{\alpha(\ell)}\}$ must form an $r$-pattern different from $P$ (as the order of the $j$-th $A$ and the $j$-th $B$ is reversed when compared to the pair $e_i,e_\ell$). So, they must form $Q$. This, however, contradicts the fact that $Q^{-j}\neq P^{-j}$.

 We thus proved that $M^{-j}\neq N^{-j}$, and consequently, by the induction hypothesis, $\tr(M^{-j})\neq \tr(N^{-j})$. Putting back digit $j$ to the original set of $k$ positions (the same for $M$ and for $N$), we infer that $\tr(M)\neq \tr(N)$. This concludes the proof of reconstructability of the family of $\{P,Q\}$-cliques, and thus the proof of Theorem~\ref{thm:random_two_patterns_gen} for mismatched pairs $P,Q$ is completed. \qed
\end{proof}



\section*{Appendix}

 Here we present two remaining proofs of the lower bounds in Theorem~\ref{thm:random_two_patterns_gen}, namely for the cases when $\{P,Q\}$ is a pair of non-collectable patterns, as well as, when it is a harmonious pair. In the former case we give a proof of a slightly more general  result, while in the latter, conversely, we present the proof in a special, but typical case. Both proofs demonstrate a common proof technique of ``implanting'' the desired structure within the vertex set of a random matching and then showing that it will actually appear with high probability.

\subsection*{Pairs of non-collectable patterns}\label{non-collect}

For  non-collectable $r$-patterns $P$ and $Q$, the (a.a.s.) upper bound  $z_{\{P,Q\}}(\RM^{(r)}_{n})=O(1)$ in Theorem~\ref{thm:random_two_patterns_gen} follows easily by a Ramsey type argument. In fact, a stronger, fully deterministic,  statement holds. Let $\cP$ be a set of  non-collectable $r$-patterns  of size $|\cP|=m\ge2$. Then the largest $\cP$-clique has size smaller than the multicolor Ramsey number $R_{m}(3)$. Indeed, let $M$ be a $\cP$-clique with $k$ edges and consider the complete graph $K_k$ whose vertices are the edges of $M$ whereas each edge of $K_k$ is colored with the pattern formed by its ends. 
By \cite[Prop. 2.1]{JCTB_paper}, there is no monochromatic triangle in this $m$-colored clique, so $k<R_{m}(3)$ by the definition of Ramsey numbers. In particular, for $m=2$ we get the bound $k\le5$. It follows that, defining
$$z_{\cP}=\max\{k:\exists\;\;\mbox{$\cP$-clique of size $k$}\},$$
we have $z_{\cP}(\RM^{(r)}_{n})\le z_{\cP}<R_{m}(3)$.
The fact that a.a.s.\ $z_{\cP}(\RM^{(r)}_{n})\ge z_{\cP}$  stems from the following  more general result, applied to $H$ being the largest $\cP$-clique possible.

\begin{proposition}\label{fixH}
Let $H$ be a fixed  ordered $r$-matching.  Then, a.a.s., there is a copy of $H$ in~$\RM^{(r)}_{n}$.
\end{proposition}

For the proof we will need a lemma from \cite{JCTB_paper}. Recall from Section~\ref{section:random} that for given disjoint subsets $A_1,\dots,A_r \subseteq [rn]$ we say that an $r$-edge $e$ \emph{spans} these sets if $|e\cap A_i|=1$ for each $1\le i\le r$.

\begin{lemma}\label{lemma:span_prob} Given integers $r\ge2$, $t\ge2r$, and $n\ge t$,
let $W_1,\dots,W_r$ be disjoint subsets of $[rn]$ such that $|W_i|=t$, for each $1\le i \le r$. Then, the probability that no edge of $\RM^{(r)}_{n}$ spans all these sets is at most $\exp\left\{{-\frac{1}{(2r)^r} \cdot \frac{t^r}{n^{r-1}}}\right\}$. \qed
\end{lemma}

\begin{proof}[Proposition~\ref{fixH}] Let $H$ have size $k$. For simplicity,  assume that $t:=n/k$ is an integer. Divide $[rn]$ into $rk$ consecutive blocks $B_1,\dots,B_{rk}$, each of length $t$. So,  $B_1=[t]$,  $B_2=\{t+1,\dots,2t\}$, etc.

We represent $H$ as a matching on vertex set $\{v_1,\dots,v_{rk}\}$, ordered consistently with the indices, disjoint from $[rn]$ and with edges $e_1,\dots,e_k$. 
For every $i=1,\dots,k$, let $I_i$ be the indicator random variable equal to 1 if there is \emph{no} edge in $\RM^{(r)}_{n}$ spanning the sets $B_j$, for all $j$ such that $v_j\in e_i$, and equal to 0 otherwise (see Figure \ref{K}). Further, define $X = \sum_{i=1}^k I_i$. We will show that a.a.s.\ $X=0$, which implies that there is a copy of $H$ in $\RM^{(r)}_{n}$.

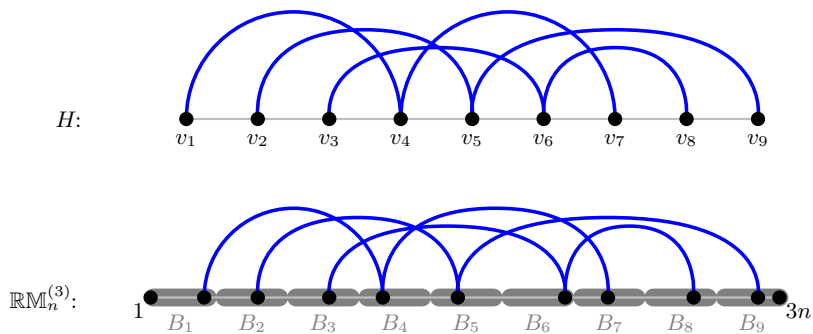
\begin{figure}[ht]
\captionsetup[subfigure]{labelformat=empty}
\begin{center}

\scalebox{0.95}
{
\centering
\begin{tikzpicture}
[line width = .5pt,
vtx/.style={circle,draw,black,very thick,fill=black, line width = 1pt, inner sep=0pt},
]

    \node[vtx] (1) at (0,0) {};
    \node[vtx] (2) at (1,0) {};
    \node[vtx] (3) at (2,0) {};
    \node[vtx] (4) at (3,0) {};
    \node[vtx] (5) at (4,0) {};
    \node[vtx] (6) at (5,0) {};
    \node[vtx] (7) at (6,0) {};
    \node[vtx] (8) at (7,0) {};
    \node[vtx] (9) at (8,0) {};
    \node[vtx] (10) at (8.8,0) {};

    \draw[line width=2.5mm, color=gray, line cap=round]  (1) -- (0.8,0) node[pos=0.5, below] {$B_1$};
    \draw[line width=2.5mm, color=gray, line cap=round]  (2) -- (1.8,0) node[pos=0.5, below] {$B_2$};
    \draw[line width=2.5mm, color=gray, line cap=round]  (3) -- (2.8,0) node[pos=0.5, below] {$B_3$};
    \draw[line width=2.5mm, color=gray, line cap=round]  (4) -- (3.8,0) node[pos=0.5, below] {$B_4$};
    \draw[line width=2.5mm, color=gray, line cap=round]  (5) -- (4.8,0) node[pos=0.5, below] {$B_5$};
    \draw[line width=2.5mm, color=gray, line cap=round]  (6) -- (5.8,0) node[pos=0.5, below] {$B_6$};
    \draw[line width=2.5mm, color=gray, line cap=round]  (7) -- (6.8,0) node[pos=0.5, below] {$B_7$};
    \draw[line width=2.5mm, color=gray, line cap=round]  (8) -- (7.8,0) node[pos=0.5, below] {$B_8$};
    \draw[line width=2.5mm, color=gray, line cap=round]  (9) -- (8.8,0) node[pos=0.5, below] {$B_9$};

    \draw[line width=0.3mm, color=lightgray]  (1) -- (10);
    \fill[fill=black, outer sep=-0.2mm]   (1) circle (0.1) node [below left] {$1$};
    \fill[fill=black, outer sep=-0.2mm]   (10) circle (0.1) node [below right] {$3n$};

    \node[vtx] (v1) at (0.75,0) {};
    \node[vtx] (v2) at (1.5,0) {};
    \node[vtx] (v3) at (2.5,0) {};
    \node[vtx] (v4) at (3.25,0) {};
    \node[vtx] (v5) at (4.3,0) {};
    \node[vtx] (v6) at (5.8,0) {};
    \node[vtx] (v7) at (6.4,0) {};
    \node[vtx] (v8) at (7.6,0) {};
    \node[vtx] (v9) at (8.5,0) {};
    \coordinate (v47) at (4.825,1.25);
    \coordinate (v36) at (4.15,1);
    \coordinate (v68) at (6.7,1);
    \coordinate (v25) at (2.9,1.12);
    \coordinate (v59) at (6.4,1.12);

    \node[color=black] at (-1.5,0) {$\RM^{(3)}_{n}\!\!:$};

    \draw[line width=0.5mm, color=blue, outer sep=2mm] (v4) arc (0:180:1.25);
    \draw[line width=0.5mm, color=blue, outer sep=2mm] plot [smooth, tension=2] coordinates {(v4) (v47) (v7)};
    \draw[line width=0.5mm, color=blue, outer sep=2mm] plot [smooth, tension=2] coordinates {(v3) (v36) (v6)};
    \draw[line width=0.5mm, color=blue, outer sep=2mm] plot [smooth, tension=2] coordinates {(v6) (v68) (v8)}; 
    \draw[line width=0.5mm, color=blue, outer sep=2mm] plot [smooth, tension=2] coordinates {(v2) (v25) (v5)}; 
    \draw[line width=0.5mm, color=blue, outer sep=2mm] plot [smooth, tension=2] coordinates {(v5) (v59) (v9)}; 

    \fill[fill=black, outer sep=1mm]   (v1) circle (0.1);
    \fill[fill=black, outer sep=1mm]   (v2) circle (0.1);
    \fill[fill=black, outer sep=1mm]   (v3) circle (0.1);
    \fill[fill=black, outer sep=1mm]   (v4) circle (0.1);
    \fill[fill=black, outer sep=1mm]   (v5) circle (0.1);
    \fill[fill=black, outer sep=1mm]   (v6) circle (0.1);
    \fill[fill=black, outer sep=1mm]   (v7) circle (0.1);
    \fill[fill=black, outer sep=1mm]   (v8) circle (0.1);
    \fill[fill=black, outer sep=1mm]   (v9) circle (0.1);

    \node[vtx] (k1) at (0.5,2.5) {};
    \node[vtx] (k2) at (1.5,2.5) {};
    \node[vtx] (k3) at (2.5,2.5) {};
    \node[vtx] (k4) at (3.5,2.5) {};
    \node[vtx] (k5) at (4.5,2.5) {};
    \node[vtx] (k6) at (5.5,2.5) {};
    \node[vtx] (k7) at (6.5,2.5) {};
    \node[vtx] (k8) at (7.5,2.5) {};
    \node[vtx] (k9) at (8.5,2.5) {};
    \coordinate (k25) at (3,3.75) {};
    \coordinate (k59) at (6.5,3.75) {};
    \coordinate (k45) at (4,3.5) {};
    \coordinate (k78) at (6.5,3.5) {};

    \draw[line width=0.3mm, color=lightgray]  (k1) -- (k9);

    \draw[line width=0.5mm, color=blue, outer sep=2mm] (k4) arc (0:180:1.5);
    \draw[line width=0.5mm, color=blue, outer sep=2mm] (k7) arc (0:180:1.5);
    \draw[line width=0.5mm, color=blue, outer sep=2mm] plot [smooth, tension=2] coordinates {(k2) (k25) (k5)};
    \draw[line width=0.5mm, color=blue, outer sep=2mm] plot [smooth, tension=2] coordinates {(k5) (k59) (k9)};
    \draw[line width=0.5mm, color=blue, outer sep=2mm] plot [smooth, tension=2] coordinates {(k3) (k45) (k6)};
    \draw[line width=0.5mm, color=blue, outer sep=2mm] plot [smooth, tension=2] coordinates {(k6) (k78) (k8)};

   \fill[fill=black, outer sep=1mm]   (k1) circle (0.1) node [below] {$v_1$};
   \fill[fill=black, outer sep=1mm]   (k2) circle (0.1) node [below] {$v_2$};
   \fill[fill=black, outer sep=1mm]   (k3) circle (0.1) node [below] {$v_3$};
   \fill[fill=black, outer sep=1mm]   (k4) circle (0.1) node [below] {$v_4$};
   \fill[fill=black, outer sep=1mm]   (k5) circle (0.1) node [below] {$v_5$};
   \fill[fill=black, outer sep=1mm]   (k6) circle (0.1) node [below] {$v_6$};
   \fill[fill=black, outer sep=1mm]   (k7) circle (0.1) node [below] {$v_7$};
   \fill[fill=black, outer sep=1mm]   (k8) circle (0.1) node [below] {$v_8$};
   \fill[fill=black, outer sep=1mm]   (k9) circle (0.1) node [below] {$v_9$};

    \node[color=black] at (-1.15,2.5) {$H\!\!:$};

\end{tikzpicture}
}

\end{center}

\caption{A matching $H$ of size 3 and its copy in $\RM^{(3)}_{n}$.}
\label{K}
			
\end{figure}

To this end, observe that for each $i$, by Lemma~\ref{lemma:span_prob} applied to the sets $B_j$, $v_j\in e_i$,	
\[
\PP(I_i=1)\le\exp\left\{-\frac{1}{(2r)^r} \cdot \frac{t^r}{n^{r-1}}\right\}=\exp\{-\Theta(n)\}.
\]
Finally, by Markov's inequality and the linearity of expectation,
\[
\PP(X\ge 1)\le
\E X = \sum_{i=1}^k \PP(I_i=1)\le k\exp\{-\Theta(n)\}=o(1),
\]
finishing the proof. \qed	
\end{proof}

Along the same lines one can get an even more general result.

\begin{proposition}\label{genH}
Let $H_n$ be a sequence of ordered $r$-matchings of size  $|H_n|\le \frac{1}{2r}\left(\frac{n}{\log n}\right)^{\frac{1}{r}}$. Then, a.a.s., there is a copy of $H_n$ in~$\RM^{(r)}_{n}$.
\end{proposition}

\subsection*{Stronger concentration inequality}

For the remaining proof we need a concentration inequality stronger than Lemma~\ref{azuma}, namely a Talagrand's concentration inequality for random permutations from \cite{Talagrand}. We quote here a slightly simplified version from \cite{LuczakMcDiarmid} (see also \cite{McDiarmid}).

\begin{lemma}[Luczak and McDiarmid \cite{LuczakMcDiarmid}]\label{tala}
	Let $h(\pi)$ be a function defined on the set of all permutations of order $N$ which, for some positive constants $c$ and $d$, satisfies
	\begin{enumerate}[label=\RMlabel]
		\item\label{thm:talagrand:i} if $\pi_2$ is obtained from $\pi_1$ by swapping two elements, then $|h(\pi_1)-h(\pi_2)|\le c$;
		\item\label{thm:talagrand:ii} for each $\pi$ and $s>0$, if $h(\pi)=s$, then in order to show that $h(\pi)\ge s$, one needs to specify only at most $ds$ values $\pi(i)$.
	\end{enumerate}
	Then, for every $\eps>0$,
	$$\PP(|h({\Pi}_N)-m|\ge \eps m)\le4 \exp(-\eps^2 m/(32dc^2)),$$
	where $m$ is the median of $h({\Pi}_N)$.
\end{lemma}

\subsection*{Lower bound for harmonious pairs}\label{fhlb}
For clarity of presentation we will give a proof in a special case when $P=(AB)^r$ and $Q=(AB)^{r_0}(BA)^{r_1}$, where $r_0+r_1=r$.
This makes notation more transparent but all the ingredients of the general proof are still needed.

Let $T_0=[r_0]$ and $T_1=[r]\setminus[r_0]=\{r_0+1,\dots,r\}$.
Further, for some $k\ge2$, let $[rk]=V_1\cup\cdots\cup V_r$ be a partition of $[rk]$ into consecutive blocks of size $|V_i|=k$, for $i=1,\dots,r$.
 Specifically, let $V_i=\{w_1^i,\dots,w_k^i\}$, $i=1,\dots,r$. Observe that for all $i<i'$ and $j\neq j'$ the pair of edges $e=\{w^h_{i}: h\in T_0\}\cup\{w^h_{j}: h\in T_1\}$ and $e'=\{w^h_{i'}: h\in T_0\}\cup\{w^h_{j'}: h\in T_1\}$ form either pattern $P$ (when $j<j'$) or pattern $Q$ (when $j>j'$) (see Figure~\ref{fig:deletion}).

 We say that an edge $e$ of the form $e=\{w_i^{h}: h\in T_0\}\cup\{w_j^{h}: h\in T_1\}$ is of \emph{type} $(i,j)$. Two types, $(i,j)$ and $(i',j')$ are called \emph{totally distinct} if $i\neq i'$ and $j\neq j'$. Thus, any set of edges with mutually totally distinct types forms a $\{P,Q\}$-clique.

 \begin{figure}[ht]
\captionsetup[subfigure]{labelformat=empty}
\begin{center}

\scalebox{0.8}
{
\begin{tikzpicture}
[line width = .5pt,
vtx/.style={circle,draw,red,very thick,fill=red, line width = 1pt, inner sep=0pt},
]

\foreach[evaluate=\i as \myi using int(4*(\i-1))] \i in {1,...,5} {
    \foreach[evaluate=\j as \myj using int(\j-1)] \j in {1,2,3,4} {
        \coordinate (c\i\j) at (0.75*\myi+0.75*\myj,0) {};
    };
};

\node[color=red] at (1.5,1.75) {$e\phantom{'}$};
\draw[line width=0.5mm, color=red, outer sep=2mm] (c21) arc (0:180:1.5);
\draw[line width=0.5mm, color=red, outer sep=2mm] (c44) arc (0:180:1.5);
\draw[line width=0.5mm, color=red, outer sep=2mm] (c54) arc (0:180:1.5);
\draw[line width=0.5mm, color=red, outer sep=2mm] (c34) arc (0:180:2.625);

\node[color=blue] at (2.25,1.75) {$e'$};
\draw[line width=0.5mm, color=blue, outer sep=2mm] (c22) arc (0:180:1.5);
\draw[line width=0.5mm, color=blue, outer sep=2mm] (c41) arc (0:180:1.5);
\draw[line width=0.5mm, color=blue, outer sep=2mm] (c51) arc (0:180:1.5);
\draw[line width=0.5mm, color=blue, outer sep=2mm] (c31) arc (0:180:1.125);

\node[color=green] at (4,1.75) {$e''$};
\draw[line width=0.5mm, color=green, outer sep=2mm] (c24) arc (0:180:1.5);
\draw[line width=0.5mm, color=green, outer sep=2mm] (c42) arc (0:180:1.5);
\draw[line width=0.5mm, color=green, outer sep=2mm] (c52) arc (0:180:1.5);
\draw[line width=0.5mm, color=green, outer sep=2mm] (c32) arc (0:180:0.75);

\foreach \i in {1,...,5} {
    \draw[fill=black, inner sep=5pt] (c\i1) circle[radius=3pt] node[below] {{\footnotesize{$w_{1}^{\i}$}}};
    \draw[fill=black, inner sep=5pt] (c\i2) circle[radius=3pt] node[below] {{\footnotesize{$w_{2}^{\i}$}}};
    \draw[fill=black, inner sep=5pt] (c\i3) circle[radius=3pt] node[label={[yshift=0mm,black]below:{\footnotesize{$\ldots$}}}] {};
    \draw[fill=black, inner sep=5pt] (c\i4) circle[radius=3pt] node[below] {{\footnotesize{$w_{k}^{\i}$}}};
};

\draw [decorate,decoration={brace,amplitude=5pt,mirror,raise=4ex}] (c11) -- (c14) node[midway,yshift=-3em]{\textcolor{black}{$V_1$}};
\draw [decorate,decoration={brace,amplitude=5pt,mirror,raise=4ex}] (c21) -- (c24) node[midway,yshift=-3em]{\textcolor{black}{$V_2$}};
\draw [decorate,decoration={brace,amplitude=5pt,mirror,raise=4ex}] (c31) -- (c34) node[midway,yshift=-3em]{\textcolor{black}{$V_3$}};
\draw [decorate,decoration={brace,amplitude=5pt,mirror,raise=4ex}] (c41) -- (c44) node[midway,yshift=-3em]{\textcolor{black}{$V_4$}};
\draw [decorate,decoration={brace,amplitude=5pt,mirror,raise=4ex}] (c51) -- (c54) node[midway,yshift=-3em]{\textcolor{black}{$V_5$}};

\draw[line width=0.3mm, color=black]  (-0.5,0) -- (14.75,0);

\end{tikzpicture}
}

\end{center}

\caption{
Here we have $r=r_0+r_1$, $r_0=2$, $r_1=3$, $P=(AB)^5=ABABABABAB$ and $Q=(AB)^2(BA)^3=ABABBABABA$ with edges $e=\{w_1^1, w_1^2,w_k^3,w_k^4, w_k^5\}$ (red), $e'=\{w_2^1, w_2^2,w_1^3,w_1^4, w_1^5\}$ (blue) and $e''=\{w_k^1, w_k^2,w_2^3,w_2^4, w_2^5\}$ (green). Edges $e,e'$ and $e,e''$ form pattern $Q$ and edges $e',e''$ form pattern $P$. Moreover, edges $e$, $e'$ and $e''$ are of types $(1,k)$, $(2,1)$ and $(k,2)$, respectively.}
\label{fig:deletion}
			
\end{figure}
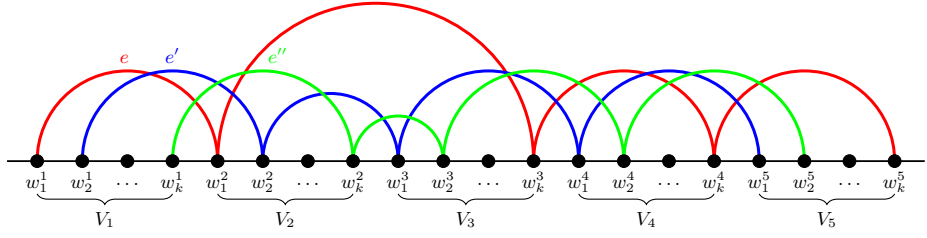	

 The idea of the proof is to blow-up the vertices $w^h_i$ to large disjoint sets, implant them within $[rn]$, and show that a.a.s.\ the random $r$-matching $\RM^{(r)}_n$ will contain  $\Theta(n^{\frac1{r-1}})$ such edges.
To this end, let us assume that $k=n^{\frac{1}{r-1}}$ and $\ell = n^{1-\frac{1}{r-1}}$ are integers, so that $k\ell=n$.
Let $W_i^h$ be the blow-up sets of vertices $w_i^h$ of size $\ell$, $i=1,\dots,k$, $h=1,\dots,r$, forming consecutive, disjoint blocks of $[rn]$.

Given an ordered pair $(i,j)\in[k]^{2}$, we say  that an edge $e\in\RM_n^{(r)}$ is a \emph{good} edge of type $(i,j)$ if $e$ spans the sets
\[
W_{i}^1,W_{i}^2,\dots,W_{i}^{r_0}, W_{j}^{r_0+1},W_{j}^{r_0+2},\dots,W_{j}^{r}.
\]
A pair of good edges $e,e'$ of types $(i,j)$ and $(i',j')$, respectively, is called \emph{separated}  if their types are totally distinct, i.e., both $i\neq i'$ and $j\neq j'$.
Note that a set of pairwise separated and good edges forms a $\{P,Q\}$-clique.

Let $X$ be the largest number of pairwise separated, good edges in $\RM_n^{(r)}$.
We will bound $\E X$ from below and afterwards apply Lemma~\ref{tala} to~$X$.

To bound  $\E X$ we will use the deletion method. Let $Y$  be the number of good edges in~$\RM_n^{(r)}$ and let $Z$ count  non-separated pairs $\{e,e'\}$ of good edges in~$\RM_n^{(r)}$.    After removing one edge from each non-separated pair, we get a set of pairwise separated, good edges of size at least $Y-Z$. Thus, $X\ge Y-Z$.

 Since there are $k^{2}$ choices for pairs $(i,j)$ and each $|W_i^h|=\ell$, we have
\[
\E Y = k^{2} \ell^r \cdot \frac{1}{\binom{rn-1}{r-1}}
\ge k^{2} \ell^r \frac{(r-1)!}{(rn)^{r-1}}
= kn^{\frac{1}{r-1}} n^{r-\frac{r}{r-1}} \frac{(r-1)!}{(rn)^{r-1}}
= \frac{(r-1)!}{r^{r-1}} k.
\]

 Let $(i,j)$ and $(i',j')$ be the types of good edges $e$ and $e'$. If $e$ and $e'$ are not separated, then either $i=i'$ or $j=j'$. Hence, for a fixed~$e$ the number of choices for $e'$ (such that $e$ and $e'$ are not separated) is at most $2k \ell^r$ and
\[
\E Z \le \frac{1}{2} k^{2} \ell^r 2 k \ell^r \cdot \frac{1}{\binom{rn-1}{r-1}\binom{rn-r-1}{r-1}}
= \E Y \cdot k \ell^r \cdot \frac{1}{\binom{rn-r-1}{r-1}}
\le \E Y \cdot (1+o(1))\frac{(r-1)!}{r^{r-1}}.
\]

 Thus, for large $n$, bounding $(1+o(1))$ by $\tfrac32$, we get
\begin{align*}
\E X \ge \E(Y-Z) = \E Y - \E Z &\ge \E Y \left(1 - \frac{3(r-1)!}{2r^{r-1}}\right)\\
 &\ge \frac{(r-1)!}{r^{r-1}} \left(1 - \frac{3(r-1)!}{2r^{r-1}}\right) k =
\Omega_r(k)
\end{align*}
for $r\ge2$.

It remains to apply Lemma~\ref{tala} to show a sharp concentration of $X$ around $\E X$. Recall (cf.\ Section~\ref{section:random}) that $\RM^{(r)}_{n}$ can be generated by using the permutation scheme. Thus, we can view $X$ as a function of an $N$-permutation $\Pi_{N}$ with $N=rn$ and $h({\Pi}_{N})=X$. Now notice that swapping two elements of $\Pi_{N}$ affects at most two edges of $\RM^{(r)}_{n}$, and so it changes the value of $h$ by at most~2. Moreover, to exhibit the event $X\ge s$, it is sufficient to reveal a set of $s$  pairwise separated, good edges. Since every such edge can be encoded by specifying $r$ values of $\Pi_{N}$, it suffices to reveal only $rs$ values of $\Pi_N$ to see a matching of size~$s$.

Let $m$ be the median of $X$. Theorem \ref{tala} (applied with $N=rn$, $c=2$, $d=r$ and $\eps=1/2$) implies that
\[
\PP(|X-m|\ge m/2)\le4\exp(-m/(512r)).
\]
Since there is a standard way to switch from the median to the expectation (see, e.g., the proof of Theorem~1.8 in~\cite{JCTB_paper}),  this completes the proof.

\subsection*{Proof of Observation~\ref{obs:1}}
 We split the proof into two cases according to the mismatch definition.

\medskip

{\bf Case 1:} Assume that $P$ and $Q$ are both collectable but do not form a harmonious pair.  Let $\lambda_P=(g_1,g_2,\dots)$ and $\lambda_Q=(h_1,h_2,\dots)$,
where $\lambda_P\neq\lambda_Q$.
If $g_1=h_1$, then take $j=1$, as, clearly, $\lambda_{P^{-1}}\neq\lambda_{Q^{-1}}$.
If (w.l.o.g.)  $g_1>h_1$, then take $j=r$, as $\lambda_{P^{-r}}\neq\lambda_{Q^{-r}}$,  unless $g_1=r$, $h_1=r-1$, and $h_2=1$, in which case  take $j=1$, as then $\lambda_{P^{-1}}\neq\lambda_{Q^{-1}}$ again.

\medskip

{\bf Case 2:} Assume that $P$ is collectable, but $Q$ is not. The case $r=3$ (and $Q=P_0=AABABB$) is the only one in which $Q^{-j}$ is collectable no matter what $j$ is (simply because all three 2-patterns are such).
Still, one can check by inspection that when choosing $j=3$ for each of the six 3-patterns beginning with $AB$, $j=1$ for $P=P_2$ and $P=P_3$, and $j=2$ for $P=P_1$, we obtain $\lambda_{P^{-j}}\neq\lambda_{Q^{-j}}$ in every case.

For $r\ge4$ we will show that, for some $j$,  $Q^{-j}$ remains non-collectable, which automatically makes the pair $P^{-j}, Q^{-j}$ a mismatch.
Being  non-collectable, $Q$ must be of the form $Q=RA^qB^tA ...$ or $Q=RB^qA^tB ...$ with $R$ collectable (possibly empty) and $q>t\ge 1$. Owing to symmetry, we assume the former option.

If $R\neq\emptyset$, take $j=1$. Then $Q^{-1}=R'A^qB^tA ...$  with $R'$ being collectable (possibly empty), thus $Q^{-1}$ is non-collectable.

For $R=\emptyset$, consider three further subcases.
 If $t\ge2$, take $Q^{-1}=A^{q-1}B^{t-1}A ...$, if $q\ge3$ and $t=1$, take $Q^{-2}=A^{q-1}BA ...$, while if $q=2$ and $t=1$, take $Q^{-r}=AABA ...$.  In each case the resulting $(r-1)$-pattern $Q^{-j}$  is non-collectable.~\qed

\end{document}